\documentclass[twoside,leqno]{article}

\usepackage[letterpaper]{geometry}
\usepackage{amsmath,amssymb,bm,mathshortcuts, color}
\usepackage[normalem]{ulem}
\usepackage{ltexpprt}
\usepackage{hyperref}
\usepackage{booktabs}
\usepackage{graphicx}
\usepackage[para]{threeparttable} 

\begin{document}

\newcommand\relatedversion{}

\title{\Large New complementarity formulations for root-finding and optimization of piecewise-affine functions in abs-normal form\relatedversion}
\author{Yulan Zhang \and Kamil A. Khan\thanks{Department of Chemical Engineering, McMaster University, Hamilton, ON, Canada. This work was supported in part by the Natural Sciences and Engineering Research Council of Canada (NSERC) under Grant
RGPIN-2017-05944.}}

\date{}

\maketitle


\fancyfoot[R]{\scriptsize{Copyright \textcopyright\ 2025 by SIAM\\
Unauthorized reproduction of this article is prohibited}}





\begin{abstract} \small\baselineskip=9pt
  Nonsmooth functions have been used to model discrete-continuous phenomena such as contact mechanics, and are also prevalent in neural network formulations via activation functions such as ReLU. At previous AD conferences, Griewank et al.~showed that nonsmooth functions may be approximated well by piecewise-affine functions constructed using an AD-like procedure. Moreover, such a piecewise-affine function may always be represented in an ``abs-normal form'', encoding it as a collection of four matrices and two vectors. We present new general complementarity formulations for root-finding and optimization of piecewise-affine functions in abs-normal form, with significantly fewer restrictions than previous approaches. In particular, piecewise-affine root-finding may always be represented as a mixed-linear complementarity problem (MLCP), which may often be simplified to a linear complementarity problem (LCP). We also present approaches for verifying existence of solutions to these problems. A proof-of-concept implementation in Julia is discussed and applied to several numerical examples, using the PATH solver to solve complementarity problems.
\end{abstract}

\section{Introduction}

Nonsmooth functions arise in several applications, such as systems with discrete transitions between flow regimes or thermodynamic phase. Nonsmoothness may also arise in numerical methods that are applied to smooth problems, such as the well-known ReLU activation function in neural networks that are used for classification or regression. In the AD2012 conference, the first approaches~\cite{KhanBartonAD,GriewankAD2012} were presented for computing correct generalized derivatives for composite functions accurately and automatically, using new extensions of the standard vector forward mode of automatic/algorithmic differentiation (AD), enabling dedicated methods for nonsmooth root-finding and optimization that rely on generalized derivative information.

Subsequently, Griewank et al.~\cite{GriewankBosse,GriewankStreubel,fiege2018algorithmic} developed a \emph{piecewise-linearization} variant of the forward AD mode. In a forward sweep through a computational graph, this variant locally approximates a composite nonsmooth function as a piecewise-affine function. Next, they proposed a new \emph{abs-normal form} for piecewise affine functions (summarized in Section~\ref{sec:ANF} below); this form encodes any piecewise-affine function as a collection of four matrices and two vectors. Again, this abs-normal form is to be constructed using a forward sweep through a computational graph, as illustrated in Example~\ref{ex:construction} below. Several properties and construction methods for the abs-normal form have been studied~\cite{streubel2014representation,kubota2018enumeration,bosse2019almost,hegerhorst2020optimality,hegerhorst2020relation}, along with assessments of the form's usefulness and efficacy~\cite{bosse2021study}.

Complementarity formulations~\cite{CottlePangStone} encode discrete behavior quite differently to nonsmooth representations, and complementarity-based solvers such as \textsf{PATH}~\cite{DirkseFerris, FerrisMunson} are powerful in practice. While a complementarity-based root-finding method was proposed for piecewise-affine functions in abs-normal form~\cite{GriewankStreubel}, we will show here that this method may be improved: its requirements may be weakened significantly by exploiting the strict lower triangularity of one of the abs-normal form's coefficient matrices. Hence, we aim to show that root-finding for the abs-normal form is equivalent to a mixed linear complementarity problem (MLCP), which in turn reduces to a linear complementarity problem (LCP) under an additional mild regularity assumption. Analogously, we will show that global optimization for a piecewise-affine function in abs-normal form is equivalent to a linear program with complementarity constraints (LPCC), which may be reformulated as a mixed-integer linear program (MILP) via a big-M approach. We also present approaches for checking whether a global minimum exists at all. These new approaches have been implemented in Julia, and we present numerical examples as well. Each of these methods is a useful application of the abs-normal form, complementing recent AD-based approaches for automatically generating abs-normal forms of local approximations of nonsmooth functions, and extending the reach of these AD variants.

\yzedit{Hegerhorst-Schultchen et al.~\cite{hegerhorst2020relation} also studied the optimization of unconstrained nonsmooth problems in a nonlinear generalization of the abs-normal form, demonstrating that these problems are equivalent to solving certain mathematical programs with equilibrium constraints (MPECs), particularly under a structural regularity assumption called the \emph{linear independence kink qualification (LIKQ}). 
Their work also extends to constrained nonsmooth nonlinear programming problems (NLPs), where both the objective function and constraints are represented in a genearlized abs-normal form. Overall, the authors explored how established concepts in MPEC theory, such as various stationarity conditions and constraint qualifications, apply to nonsmooth optimization problems in abs-normal form with embedded nonlinearities. Our new optimization formulation~\eqref{eq:LPCC}, on the other hand, focuses strictly on piecewise affine functions in abs-normal form but does not impose other structural assumptions. In particular, we do not consider constraint qualifications or stationarity conditions at all, although our piecewise affine requirement may be seen to play a similar role to a constraint qualification.}

The remainder of this manuscript is structured as follows. Section~\ref{sec:preliminaries} introduces abs-normal forms and an established complementarity-based root-finding approach. Section~\ref{sec:root} presents our new root-finding formulations for piecewise-affine functions, and Section~\ref{sec:optim} presents analogous new optimization formulations. Section~\ref{sec:implementation} summarizes our proof-of-concept Julia implementation (which is also available on GitHub), and Section~\ref{sec:examples} applies this implementation to several illustrative numerical examples.

\section{Preliminaries}
\label{sec:preliminaries}

\subsection{Notation}

For any vectors $\mathbf{p},\mathbf{q}\in\reals^n$, $|\mathbf{p}|$ will denote the componentwise absolute value, whose $\supth{i}$ component is $|p_i|$, and $\max(\mathbf{p},\mathbf{q})$ will denote an analogous componentwise bivariate ``$\max$'', whose $\supth{i}$ component is $\max(p_i,q_i)$. Inequalities involving vectors are to be interpreted componentwise as well.

This article is concerned with continuous \emph{piecewise-affine} (``$\mathcal{PA}$'') functions
$\f:\reals^n\to\reals^m$, whose values are chosen from finitely many affine functions (called \emph{affine pieces}). Several properties of piecewise affine functions were established by Scholtes~\cite{Scholtes}. If each affine piece of a $\mathcal{PA}$ function
$\f$ is in fact linear, then $\f$ is \emph{piecewise linear}~($\mathcal{PL}$).

The \emph{complementarity condition} $\0\leq\mathbf{\p}\perp\mathbf{\q}\geq\0$ is equivalent to the following three conditions holding simultaneously:
$\p\geq\0$, $\mathbf{q}\geq\0$, and $\transpose{\p}\mathbf{q}=0$. 

In optimization problems, constraints are marked with the phrase ``subject to'', abbreviated to ``s.t.''.

\subsection{Abs-normal forms}
\label{sec:ANF}

Any $\mathcal{PA}$ mapping  may in principle be
represented as a composition of affine functions and absolute-value
functions~\cite{Scholtes}. This observation motivated Griewank et al.'s development of
\emph{abs-normal forms}~\cite{GriewankStreubel,GriewankBosse},
which are described as follows, and effectively encode $\mathcal{PA}$
functions as constant vector-matrix pairs.

\begin{Definition}{\rm [adapted from~\cite{GriewankStreubel,GriewankBosse}]
  \label{def:ANF}
  Given a $\mathcal{PA}$ function $\f:\reals^n\to\reals^m$, there exist a number $s\in\mathbb{N}$,
  matrices $\Z\in\reals^{s\times n}$,
  $\J\in\reals^{m\times n}$, and $\Y\in\reals^{m\times s}$, a strictly lower triangular matrix $\LL\in\reals^{s\times s}$, and vectors $\cc \in\reals^s$ and $\bb\in\reals^m$, for which, for any $\x\in\reals^n$, there exists a unique vector $\yzedit{\z \in \reals^s}$ for which
\begin{equation}
\label{eq:ANF}
\begin{bmatrix}
  \z \\ \f(\x)
\end{bmatrix}
=
\begin{bmatrix}
  \cc \\ \bb
\end{bmatrix}
+
\begin{bmatrix}
  \Z & \LL \\ \J & \Y
\end{bmatrix}
\begin{bmatrix}
  \x \\ |\z|
\end{bmatrix},
\end{equation}
The right-hand side of~\eqref{eq:ANF} is called an \emph{abs-normal form} of $\f$. }
\end{Definition}

Since $\LL$ is strictly lower triangular, observe that for each $i\in\{1,\ldots,s\}$, the $\supth{i}$ row of \eqref{eq:ANF} defines $z_i$ in terms of $\x$ and values $|z_j|$ with $j<i$.

 Griewank et al.~\cite{GriewankStreubel} outlined how an abs-normal form may be constructed for a $\mathcal{PA}$ function that is already expressed as a composition of affine functions and absolute-value operations.  We illustrate this procedure in the following example.

\begin{Example}\label{ex:construction}
  Consider the $\mathcal{PA}$ mapping $f:x\in\reals\mapsto x + \Big|2|3x + 4| - 5\Big| + 6|7x-8|$. To construct an abs-normal form for $f$, we first construct a vector $\z$ whose components are the arguments of the absolute-value operations in $f$'s definition, listed in the order these would be evaluated by a compiler or forward-AD-like procedure:
  \begin{align*}
    z_1 &:= 3x + 4, \\
    z_2 &:= 2|3x+4| - 5 = 2|z_1|-5, \\
    z_3 &:= 7x-8.
  \end{align*}
  Then, $f$ is written in terms of $x$ and $|\z|$. For each $x\in\reals$,
  \[
    f(x) = x + |z_2| + 6|z_3|.
  \]
  At this point, the above equations may be collected into the following abs-normal form:
  \[
    \begin{bmatrix}
      z_1 \\ z_2 \\ z_3 \\ f(x)
    \end{bmatrix}
    =
    \begin{bmatrix}
      \phantom{-}4 \\ -5 \\ -8 \\ \phantom{-}0
    \end{bmatrix}
    +
    \begin{bmatrix}
      3 & 0 & 0 & 0 \\
      0 & 2 & 0 & 0 \\
      7 & 0 & 0 & 0 \\
      1 & 0 & 1 & 6
    \end{bmatrix}
    \begin{bmatrix}
      x \\ |z_1| \\ |z_2| \\ |z_3|
    \end{bmatrix}.
  \]
\yzedit{By inspection, this representation is consistent with \eqref{eq:ANF} with the identifications:
\begin{align*}
    \cc&= \begin{bmatrix}
    \phantom{-}4 \\ -5 \\ -8
    \end{bmatrix},
    & \Z&= \begin{bmatrix}
    3 \\ 0 \\ 7
    \end{bmatrix},
    & \LL&= \begin{bmatrix}
    0 & 0 & 0\\
    2 & 0 & 0 \\
    0 & 0 & 0
    \end{bmatrix},\\
    \bb&= \begin{bmatrix}
    0
    \end{bmatrix}, 
    &\J&= \begin{bmatrix}
    1
    \end{bmatrix},
    &\Y&= \begin{bmatrix}
    0 & 1 & 6
    \end{bmatrix}.
\end{align*}
Observe that $\LL$ is indeed strictly lower triangular, since in general, the above procedure will never define any $z_i$ in terms of another $z_j$ with $j>i$.}
\end{Example}

In the remainder of this manuscript, we aim to develop new complementarity-based methods for root-finding and optimization for the generic function in the following assumption.

\begin{Assumption}
  \label{ass:ANF}
  Consider a $\mathcal{PA}$ function $\f:\reals^n\to\reals^m$, with the
  abs-normal form~\eqref{eq:ANF}.
\end{Assumption}

To solve the problem of finding a root of $\f$, the following
complementarity-based approach was established
in~\cite[Section~7.5]{GriewankStreubel}. 
\yzedit{In the following proposition, a \emph{linear complementarity problem (LCP)} involves solving a condition of the form $\0\leq \x \perp\mathbf{Mx}+\mathbf{q}\geq \0$ for $\x$, and a \emph{mixed linear complementarity problems (MLCP)} is an LCP coupled with linear equation systems. Established solvers including \textsf{PATH}~\cite{DirkseFerris, FerrisMunson} can solve LCPs and variants including certain MLCPs.}

\begin{proposition}[from~\cite{GriewankStreubel}]
  \label{prop:oldRoot}
Suppose that Assumption~\ref{ass:ANF} holds, and assume
that $m=n$ and $\J$ is nonsingular. Define
the following quantities:
\[
\bS:=\LL-\Z\J^{-1}\Y\in\reals^{s\times s},
\qquad
\hat{\cc}:=\cc-\Z\J^{-1}\bb\in\reals^s.
\]
Then $\f(\x)=\0$ if and only if there exist $\uu,\w \in\reals^s$ which solve the following MLCP:
\begin{equation}
  \label{eq:oldMLCP}
  \begin{array}{rcl}
    \uu-\w&=&\hat{\cc} + \bS(\uu+\w) \\
    \0\leq \uu &\perp& \w\geq \0.
  \end{array}
\end{equation}
Moreover, if $(\I-\bS)$ is nonsingular, then $\w$ solves the following LCP:
\begin{align}
    \label{eq:oldLCP}
      \0\leq (\I-\bS)^{-1}\hat{\cc} + (\I-\bS)^{-1}(\I+\bS)\w\perp \w\geq \0,
\end{align}
with $\uu$ still described by~\eqref{eq:oldMLCP}.
\end{proposition}

\section{New root-finding formulations}
\label{sec:root}

The derivation of Proposition~\ref{prop:oldRoot} does not actually make use of the fact that $\LL$ is strictly lower triangular. In this section, we will show that this triangularity allows for similar
root-finding formulations to Proposition~\ref{prop:oldRoot} under significantly less restrictive assumptions.  Observe that $(\I-\LL)$
is unit lower triangular, and is therefore invertible. Moreover, $(\I-\LL)^{-1}$
is straightforward to evaluate.  The matrices and vectors in the
following definition will be useful in our subsequent development; these quantities all have the same dimensions as their namesakes.
\begin{Definition}
\label{def:MandV}
  Under Assumption~\ref{ass:ANF}, define the following matrices and vectors:
  \begin{align*}
    \tilde{\cc}&:=(\I-\LL)^{-1}\cc, &
                                      \tilde{\bb}&:= \bb+\Y\tilde{\cc} \\
    \tilde{\LL}&:= (\I-\LL)^{-1}(\I+\LL),
                                    &\tilde{\Y}&:= \Y(\I+\tilde{\LL}), \\
    \tilde{\Z}&:= (\I-\LL)^{-1}\Z,
    &\tilde{\J}&:= \J+\Y\tilde{\Z}.
  \end{align*}
  In the special case where $\tilde{\J}$ is nonsingular, define the
  following quantities as well:
  \[
    \check{\cc}:=\tilde{\cc} - \tilde{\Z}\tilde{\J}^{-1}\tilde{\bb},
    \qquad
    \check{\bS}:=\tilde{\LL} - \tilde{\Z}\tilde{\J}^{-1}\tilde{\Y}.
  \]
\end{Definition}
Since $\LL$ is strictly lower triangular, the matrix $\tilde{\LL}$ is unit lower triangular.

\yzedit{
We next illustrate these constructions for the function from Example~\ref{ex:construction}.
  
\begin{Example}\label{ex:ConstructMartixAndVector}
    Consider the  $\mathcal{PA}$ mapping $f:x\in\reals\mapsto x + \Big|2|3x + 4| - 5\Big| + 6|7x-8|$, whose abs-normal form~\eqref{eq:ANF} was constructed in Example~\ref{ex:construction}. Under Definition~\ref{def:MandV}, and using the matrices obtained in Example~\ref{ex:construction}, the following matrices and vectors may be computed directly:
    \begin{align*}
    \tilde{\cc}&= \begin{bmatrix}
    \phantom{-}4 \\ \phantom{-}3 \\ -8
    \end{bmatrix},
    & \tilde{\Z}&= \begin{bmatrix}
    3 \\ 6 \\ 7
    \end{bmatrix},
    & \tilde{\LL}&= \begin{bmatrix}
    1 & 0 & 0\\
    4 & 1 & 0 \\
    0 & 0 & 1
    \end{bmatrix},\\
    \tilde{\bb}&= \begin{bmatrix}
    -45
    \end{bmatrix}, 
    &\tilde{\J}&= \begin{bmatrix}
    49
    \end{bmatrix},
    &\tilde{\Y}&= \begin{bmatrix}
    4 & 2 & 12
    \end{bmatrix}.
\end{align*}
In this example, $\tilde{\J}$ is indeed nonsingular, so we may also compute the remaining quantities:
    \begin{align*}
    \check{\cc}&=\begin{bmatrix}
    \phantom{-}6.8 \\ \phantom{-}8.5 \\ -1.6
    \end{bmatrix}, 
    &\check{\bS}&=\begin{bmatrix}
    \phantom{-}0.76 & -0.12 &-0.74\\
    \phantom{-}3.50  & \phantom{-}0.76 & -1.50\\
    -0.57 & -0.28 & -0.71
    \end{bmatrix}.
    \end{align*}
The above matrices and vectors can be used to numerically construct our new formulations for the function $f(x)= x + \Big|2|3x + 4| - 5\Big| + 6|7x-8|$ in this manuscript. 

\end{Example}}

The following theorem uses the constructions of Definition~\ref{def:MandV} in a new general root-finding result.

\begin{theorem}
  \label{thm:newMLCP}
  Suppose that Assumption~\ref{ass:ANF} holds, and consider the auxiliary matrices and vectors from Definition~\ref{def:MandV}.
  Then, $\f(\x)=\0$ if and only if there exists $\w\in\reals^s$ for which $(\x,\w)$ solves the following MLCP:
  \begin{equation}
    \label{eq:newMLCP}
    \begin{array}{rcl}
      \0 &=& \tilde{\bb} + \tilde{\J}\x + \tilde{\Y}\w \\
      \0\leq \w &\perp& \tilde{\cc} + \tilde{\Z}\x + \tilde{\LL}\w \geq \0.
    \end{array}
  \end{equation}
  Moreover, if $m=n$ and $\tilde{\J}$ is nonsingular,
  then $\f(\x)=\0$ if and only if 
  \begin{equation}
  \label{eq:linear}
    \tilde{\J}\x=-\tilde{\bb}-\tilde{\Y}\w,
  \end{equation}
  where $\w$ solves the LCP:
  \begin{equation}
    \label{eq:newLCP}
    \0\leq \w \perp \check{\cc} + \check{\bS}\w \geq \0.
  \end{equation}
\end{theorem}

\begin{proof}
  Consider some fixed $\x\in\reals^n$.

  For the ``only if'' part of the theorem's claim, consider the unique vector $\z$ that is consistent with \eqref{eq:ANF}, and define $\uu:=\max(\0,\z)$ and $\w:=\max(\0,-\z)$. It follows that $\z = \uu-\w$, $|\z|=\uu + \w$, and  $\0\leq\uu\perp\w\geq\0$.
  Thus, the top block row of \eqref{eq:ANF} becomes $\uu-\w = \cc + \Z\x + \LL(\uu+\w)$,
  and so
  \begin{equation}
  \label{eq:u}
    \uu=\tilde{\cc} + \tilde{\Z}\x + \tilde{\LL}\w.
  \end{equation}

  Beginning with \eqref{eq:ANF}, using the definitions of $\uu$ and $\w$ to eliminate $\z$ throughout, and then using \eqref{eq:u} to eliminate $\uu$, we conclude that $\w$  satisfies the following conditions simultaneously:
  \begin{align*}
    \f(\x) &= \tilde{\bb} + \tilde{\J}\x + \tilde{\Y}\w,\\
    \0 \leq \w &\perp \tilde{\cc} + \tilde{\Z}\x + \tilde{\LL}\w \ge \0.
  \end{align*}
  Hence, if $\f(\x)=\0$, then $(\x,\w)$ solves the MLCP~\eqref{eq:newMLCP} as claimed.

  Next, for the ``if'' part of the theorem's claim, suppose that $(\x,\w)$ solves the MLCP~\eqref{eq:newMLCP}. Define $\uu$ by \eqref{eq:u}, so that \eqref{eq:newMLCP} yields $\0\leq\uu\perp\w\geq\0$. Define $\z:=\uu-\w$; it follows that $\uu=\max(\0,\z)$ and $\w=\max(\0,-\z)$, and so $|\z|=\uu+\w$. Now, \eqref{eq:u} can be rearranged to yield $\uu-\w = \cc + \Z\x + \LL(\uu+\w)$, and so $\z = \cc + \Z\x + \LL|\z|$. Hence, $\z$ is the unique vector that is consistent with \eqref{eq:ANF}, and so
  \begin{align*}
   & \f(\x) \\
    &\quad= \bb + \J\x+\Y|\z| \\
    &\quad= \bb + \J\x + \Y(\uu+\w) \\
    &\quad= \tilde{\bb} + \tilde{\J}\x + \tilde{\Y}\w \\
    &\quad = \0,
  \end{align*}
as claimed.
  
  Next, if $m=n$ and $\tilde{\J}$ is nonsingular, then \eqref{eq:linear} uniquely specifies $\x$ as:
  \[
    \x = -\tilde{\J}^{-1}\tilde{\bb} - \tilde{\J}^{-1}\tilde{\Y}\w.
  \]
  Substituting this expression for $\x$ into the complementarity condition in the MLCP~\eqref{eq:newMLCP}, and noting that \eqref{eq:linear} is equivalent to the equation in \eqref{eq:newMLCP},  we obtain \eqref{eq:newLCP} as required.
\end{proof}

\yzedit{For instance, the following MLCP instance was constructed according to~\eqref{eq:newMLCP} for the function $f$ from Examples~\ref{ex:construction} and~\ref{ex:ConstructMartixAndVector}, and so its solutions provide roots of $f$.
\begin{align*}
0 &= -45 +  49 x + \begin{bmatrix}
    4 & 2 & 12
\end{bmatrix}\begin{bmatrix}
      w_1 \\ w_2 \\ w_3
  \end{bmatrix}, \\
  \0\leq \begin{bmatrix}
      w_1 \\ w_2 \\ w_3
  \end{bmatrix} &\perp \begin{bmatrix}
    \phantom{-}4\\ \phantom{-}3 \\ -8
  \end{bmatrix} + \begin{bmatrix}
    3 \\ 6 \\ 7
  \end{bmatrix}[x] + \begin{bmatrix}
      1 & 0 & 0 \\
      4 & 1 & 0\\
      0 & 0 & 1 
  \end{bmatrix}\begin{bmatrix}
      w_1 \\ w_2 \\ w_3
  \end{bmatrix} \geq \0.
\end{align*}
Alternatively, since $\tilde{\J}$ is nonsingular in this example, we can obtain a root of $f$ by solving:
\begin{equation*}
    49x = 45 -\begin{bmatrix}
        4 & 2 & 12
    \end{bmatrix}\begin{bmatrix}
        w_1 \\ w_2 \\ w_3
    \end{bmatrix},
\end{equation*}
where $\w$ solves the following LCP, derived from our new root-finding approach~\eqref{eq:newLCP}:
\begin{equation*}
    \0\leq \begin{bmatrix}
        w_1 \\ w_2 \\ w_3
    \end{bmatrix} \perp \begin{bmatrix}
        \phantom{-}6.8 \\ \phantom{-}8.5 \\ -1.6
    \end{bmatrix} + \begin{bmatrix}
        \phantom{-}0.76 & -0.12 &-0.74\\
        \phantom{-}3.50  & \phantom{-}0.76 & -1.50\\
        -0.57 & -0.28 & -0.71
    \end{bmatrix}\begin{bmatrix}
        w_1 \\ w_2 \\ w_3
    \end{bmatrix} \geq \0.
\end{equation*}}

Unlike~\eqref{eq:oldMLCP}, the MLCP~\eqref{eq:newMLCP} does not require $m=n$, and is equivalent to problem of solving $\f(\x)=\0$ for $\x$, with no additional requirements. For example, observe that $\J$ will be $\0$ if $\f$ is a linear combination of absolute-value terms, yet this situation violates the requirements of Proposition~\ref{prop:oldRoot}. Griewank et al.~\cite{GriewankStreubel} note that this particular situation may be rectified by a reformulation that is nontrivial to carry out, and results in a new abs-normal form with significantly larger coefficient matrices. However, this is unnecessary in our new MLCP~\eqref{eq:newMLCP}, which has no additional requirements at all. 

Moreover, if
$\tilde{\J}$ is square and nonsingular, then a root $\x$ of $\f$ may be determined by solving the
LCP~\eqref{eq:newLCP} followed by the linear equation
system~\eqref{eq:linear}. Indeed, in our experience, it seems that $\tilde{\J}$ is nonsingular whenever each term $|z_i|$ is meaningfully used in the construction of $\f$, though this notion seems difficult to formalize.

As the MLCP~\eqref{eq:newMLCP} is a
specialization of the
\emph{mixed complementarity problem} (MCP) discussed by Dirkse and Ferris~\cite{DirkseFerris}, \eqref{eq:newMLCP} may be approached numerically using the \textsf{PATH} solver~\cite{DirkseFerris, FerrisMunson} \yzedit{which implements a stabilized Newton method and nominally assumes the nonsingularity of the Jacobian matrix. The solver exhibits global convergence to an MCP solution when this nonsingularity condition is met.} \textsf{PATH} may also be applied to solve the LCP~\eqref{eq:newLCP} when $\tilde{\J}$ is square and nonsingular.
The following corollary of Theorem~\ref{thm:newMLCP} is immediate; definitions and properties of Q-matrices and P-matrices are
established in~\cite{CottlePangStone}.

\begin{corollary}
  Suppose that Assumption~\ref{ass:ANF} holds with $m=n$, and suppose the matrix $\tilde{\J}$ from Definition~\ref{def:MandV} is nonsingular. If $\check{\bS}$ is a Q-matrix, then there exists $\x\in\reals^n$ for
  which $\f(\x)=\0$.  If $\check{\bS}$ is a P-matrix, then there is exactly
  one such $\x$.
\end{corollary}

Observe that the above corollary places no constraints on $\bb$ or $\cc$, as these vectors are
not involved in the definitions of $\tilde{\J}$ or $\check{\bS}$.  

Griewank et al.~\cite{GriewankStreubel} emphasize the relative simplicity of the case in which $\LL=\0$; here  $\f$ is said to be \emph{simply switched}~\cite{GriewankStreubel}. The following corollary is the special case of Theorem~\ref{thm:newMLCP}
 for a simply switched function $f$. 

\begin{corollary}
  Suppose that Assumption~\ref{ass:ANF} holds with $m=n$, and that $\LL=\0$.  Then,
  $\f(\x)=\0$ if and only if there exists $\w\in\reals^s$ for which
  $(\x,\w)$ solves the following MLCP:
  \begin{equation*}
    \begin{array}{rcl}
      \0 &=& (\bb+\Y\cc) + (\J+\Y\Z)\x + 2\Y\w \\
      \0\leq \w &\perp& \cc + \Z\x + \w \geq \0.
    \end{array}
  \end{equation*}
\end{corollary}


\section{New optimization formulations}
\label{sec:optim}

This section adopts a similar approach to the previous section in order to optimize $\mathcal{PA}$ functions with known abs-normal forms.
The following result shows that minimizing a function in abs-normal
form is equivalent to solving a \emph{linear program with complementarity
constraints (LPCC)} in the sense of \cite{HuMitchell, FukushimaPang, YuMitchellPang}.  This
approach complements the approach taken in 
\cite[Sections~4.1--4.2]{HuMitchell}, in which LPCC formulations are
provided for $\mathcal{PA}$ functions that satisfy certain structural
assumptions and have known affine pieces on known polyhedral subdomains. 

\subsection{Minimizing a function in abs-normal form}

\begin{theorem}
  \label{thm:LPCC}
  Suppose that Assumption~\ref{ass:ANF} holds with $m=1$, and consider the auxiliary matrices and vectors from Definition~\ref{def:MandV}.  A point
  $\x^\ast\in\reals^n$ is a (global) minimum of $f$ if and only if there
  exists $\w^\ast\in\reals^s$ for which $(\x^\ast,\w^\ast)$ solves the following LPCC:
  \begin{equation}
    \label{eq:LPCC}
    \begin{array}{cl}
      \displaystyle \min_{\x\in\reals^n,\w\in\reals^s} & \tilde{b} + \tilde{\J}\x + \tilde{\Y}\w \medskip\\
      \mathrm{s.t.} & \0 \leq \w\perp\tilde{\cc} + \tilde{\Z}\x + \tilde{\LL}\w\geq \0.
    \end{array}
  \end{equation}
\end{theorem}
\begin{proof}
  As in the proof of Theorem~\ref{thm:newMLCP}, for each $\x\in\reals^n$, there is a unique vector $\w(\x)\in\reals^s$ that simultaneously satisfies:
  \begin{align*}
    f(\x) &= \tilde{b} + \tilde{\J}\x + \tilde{\Y}\w(\x),\\
    \0 \leq \w(\x) &\perp \tilde{\cc} + \tilde{\Z}\x + \tilde{\LL}\w(\x) \ge \0.
  \end{align*}

  For the ``only if'' claim of the theorem, suppose that $\x^\ast$ is a minimum of $f$,  choose $\z^\ast$ consistently with \eqref{eq:ANF} (with $\x^\ast$ in place of $\x$), and define $\w^\ast:=\max(\0,-\z^\ast)$. Proceeding similarly to the proof of Theorem~\ref{thm:newMLCP}, it follows that $(\x^\ast,\w^\ast)$ is feasible in \eqref{eq:LPCC}, with a corresponding objective value of $f(\x^\ast)$. Now, consider an arbitrary fixed $\x\in\reals^n$. If $\w$ satisfies $\0\leq \w\perp\tilde{\cc} + \tilde{\Z}\x + \tilde{\LL}\w\geq \0$, and if we define $\uu:=\tilde{\cc} + \tilde{\Z}\x + \tilde{\LL}\w$ and $\z:=\uu-\w$, it follows that $|\z|=\uu+\w$, and so the definition of $\uu$ yields:
  \[
    \z = \cc+\Z\x + \LL|\z|.
  \]
  Thus, $\z$ is consistent with \eqref{eq:ANF}, and so $f(\x)=b+\J\x+\Y|\z|=\tilde{b} + \tilde{\J}\x + \tilde{\Y}\w$. Hence, allowing $\x$ to vary, we have:
  \begin{align*}
    &\min_{\x\in\reals^n,\w\in\reals^s}\, \tilde{b} + \tilde{\J}\x + \tilde{\Y}\w \\ &\qquad\qquad\qquad\mathrm{s.t.} \quad \0 \leq \w\perp\tilde{\cc} + \tilde{\Z}\x + \tilde{\LL}\w\geq \0 \\
    &\qquad\geq \min_{\x\in\reals^n}\,\left[\inf_{\w\in\reals^s} \tilde{b} + \tilde{\J}\x + \tilde{\Y}\w \right.\\
    &\qquad\qquad\quad\qquad\left.\mathrm{s.t.} \quad \0 \leq \w\perp\tilde{\cc} + \tilde{\Z}\x + \tilde{\LL}\w\geq \0 \right] \\
    &\qquad = \min_{\x\in\reals^n} f(\x)\\
      &\qquad = f(\x^\ast),
  \end{align*}
  which has also been established as an attainable objective value for \eqref{eq:LPCC}. Hence $(\x^\ast,\w^\ast)$ solves \eqref{eq:LPCC} as claimed.

  For the ``if'' claim of the theorem, suppose that $(\x^\ast,\w^\ast)$ solves \eqref{eq:LPCC}.  Define $\uu^\ast$ by \eqref{eq:u} (with $\x^\ast$ in place of $\x$ and $\w^\ast$ in place of $\w$), so that the constraint of \eqref{eq:LPCC} yields $\0\leq\uu^\ast\perp\w^\ast\geq\0$. Define $\z^\ast:=\uu^\ast-\w^\ast$; it follows that $\uu^\ast=\max(\0,\z^\ast)$ and $\w^\ast=\max(\0,-\z^\ast)$, and so $|\z^\ast|=\uu^\ast+\w^\ast$. Now, \eqref{eq:u} can be rearranged to yield $\uu^\ast-\w^\ast = \cc + \Z\x^\ast + \LL(\uu^\ast+\w^\ast)$, and so $\z^\ast = \cc + \Z\x^\ast + \LL|\z^\ast|$. Hence, $\z^\ast$ is the unique vector that is consistent with \eqref{eq:ANF} in place of $\z$ when $\x$ is replaced by $\x^\ast$. Thus,
  \begin{align*}
    &f(\x^\ast) \\
    &\quad= b + \J\x^\ast+\Y|\z^\ast| \\
    &\quad= b + \J\x^\ast + \Y(\uu^\ast+\w^\ast) \\
    &\quad= \tilde{b} + \tilde{\J}\x^\ast + \tilde{\Y}\w^\ast,
  \end{align*}
  and so $f(\x^\ast)$ is the optimal objective value of \eqref{eq:LPCC}. Now consider an arbitrary $\x\in\reals^n$, choose $\z$ consistently with \eqref{eq:ANF}, and define $\uu:=\max(\0,\z)$ and $\w:=\max(\0,-\z)$. As in the proof of Theorem~\ref{thm:newMLCP}, it follows that $f(\x)=\tilde{b} + \tilde{\J}\x + \tilde{\Y}\w$ and $\0 \leq \w \perp \tilde{\cc} + \tilde{\Z}\x + \tilde{\LL}\w \ge \0$. Hence, $(\x,\w)$ is feasible in \eqref{eq:LPCC}, and $f(\x)$ is the corresponding objective value. Since $(\x^\ast,\w^\ast)$ solves \eqref{eq:LPCC} with a corresponding \emph{optimal} objective value of $f(\x^\ast)$, it follows that $f(\x^\ast)\leq f(\x)$, as required.
\end{proof}

\yzedit{For instance, the following LPCC system uses our new optimization formulation~\eqref{eq:LPCC} to globally minimize the $\mathcal{PA}$ function $f$ from Example~\ref{ex:construction}, we construct , along with the matrices and vectors established in Example~\ref{ex:ConstructMartixAndVector}:
\begin{align*}
\displaystyle &\min_{x\in\reals,\w\in\reals^3}\; -45 +  49 x + \begin{bmatrix}
    4 & 2 & 12
\end{bmatrix}\begin{bmatrix}
      w_1 \\ w_2 \\ w_3
  \end{bmatrix} \\
  &\mathrm{s.t.}\; \0\leq \begin{bmatrix}
      w_1 \\ w_2 \\ w_3
  \end{bmatrix} \perp \begin{bmatrix}
    \phantom{-}4\\ \phantom{-}3 \\ -8
  \end{bmatrix} + \begin{bmatrix}
    3 \\ 6 \\ 7
  \end{bmatrix}[x] + \begin{bmatrix}
      1 & 0 & 0 \\
      4 & 1 & 0\\
      0 & 0 & 1 
  \end{bmatrix}\begin{bmatrix}
      w_1 \\ w_2 \\ w_3
  \end{bmatrix} \geq \0.
\end{align*}}

 A numerical method for solving LPCCs is described
in~\cite{YuMitchellPang}. Alternatively, if bounds on $\x^\ast$ and
$\w^\ast$ are known, then, in the spirit of~\cite[Theorem~1]{Pardalos},
the above LPCC may be reformulated as a mixed-integer
linear program (MILP). 

\begin{corollary}
  Suppose that Assumption~\ref{ass:ANF} holds with $m=1$. Suppose some value $\mu>0$ is known, such that any solution $(\x^\dagger,\w^\dagger)$ of~\eqref{eq:LPCC}
  must satisfy both $\w^\dagger\leq \mu \e$ and $\tilde{\cc} + \tilde{\Z}\x^\dagger +
  \tilde{\LL}\w^\dagger \leq \mu \e$, where $\e\in\reals^m$ denotes a vector of ones.
  Then, $\x^\ast \in \reals^n$ is a minimum of $f$ if and only if, for some $\w^\ast$
  and $\y^\ast$, $(\x^\ast,\w^\ast,\y^\ast)$ solves the following MILP:
  \begin{equation}
    \label{eq:MILP}
    \begin{array}{cl}
      \displaystyle \min_{\x,\w,\y}& \tilde{b} + \tilde{\J}\x + \tilde{\Y}\w
                                \medskip\\
      \mathrm{s.t.} & \0 \leq \w \leq \mu \y, \\
      & \0 \leq \tilde{\cc} + \tilde{\Z}\x + \tilde{\LL}\w \leq \mu(\e-\y), \\
      & \y\in\{0,1\}^s.
    \end{array}
  \end{equation}
\end{corollary}
\begin{proof}
Under the hypotheses of this corollary, \eqref{eq:MILP} is a standard ``big-M'' reformulation of the LPCC~\eqref{eq:LPCC}; the $\supth{i}$ component of the binary vector $\y$  encodes which of $w_i$ or $[\tilde{\cc}+\tilde{\Z}\x+\tilde{\LL}\w]_i$ is zero under the complementarity constraint of \eqref{eq:MILP}.
\end{proof}

\subsection{Existence of minima}

As defined in~\cite{RockafellarWets}, \emph{horizon functions} will be 
shown to characterize the existence of unconstrained global minima for
piecewise affine functions. First we establish some basic properties of horizon functions for $\mathcal{PA}$ functions, and then we use these properties to obtain an MLCP that characterizes existence of a global minimum.

\begin{Definition}[adapted from \cite{RockafellarWets}]
  Given a continuous function $f:\reals^n\to\reals$, the \emph{horizon
  function} for $f$ is the mapping:
\begin{align*}
  &f^\infty:\reals^n\to\bar{\reals}: \\
  &\quad \bm\xi\mapsto \lim_{\delta\downTo{0}} 
    \left(\inf\{\lambda f(\tfrac{\x}{\lambda}):
      \quad 0<\lambda<\delta,\quad \|\x-\bm\xi\|<\delta\}\right).
\end{align*}
\end{Definition}

Any horizon function $f^\infty$ is positively
homogeneous~\cite[Theorem~3.2.1]{RockafellarWets}; thus,
$f^\infty(0)=0$ and
$\inf_{\bm\xi\in\reals^n}f^\infty(\bm\xi)\in\{-\infty,0\}$.  \yzedit{Intuitively, a horizon function describes how $f(\x)$ behaves when $\|\x\|$ is large, and the graph of $f^\infty$ is essentially a ``highly zoomed-out'' variant of the graph of $f$. Figure~\ref{fig:horizon} illustrates the horizon function of a particular $\mathcal{PA}$ function:
  \begin{equation}
    \label{eq:horizon}
    f(x) = x + \Big|2|x - 2| - 10\Big| + 20.
  \end{equation}} The following lemma relates horizon function behavior to the existence of a global minimum
when $f$ is $\mathcal{PA}$.
\begin{figure}[h]
    \centering
    \includegraphics[width=0.5\linewidth]{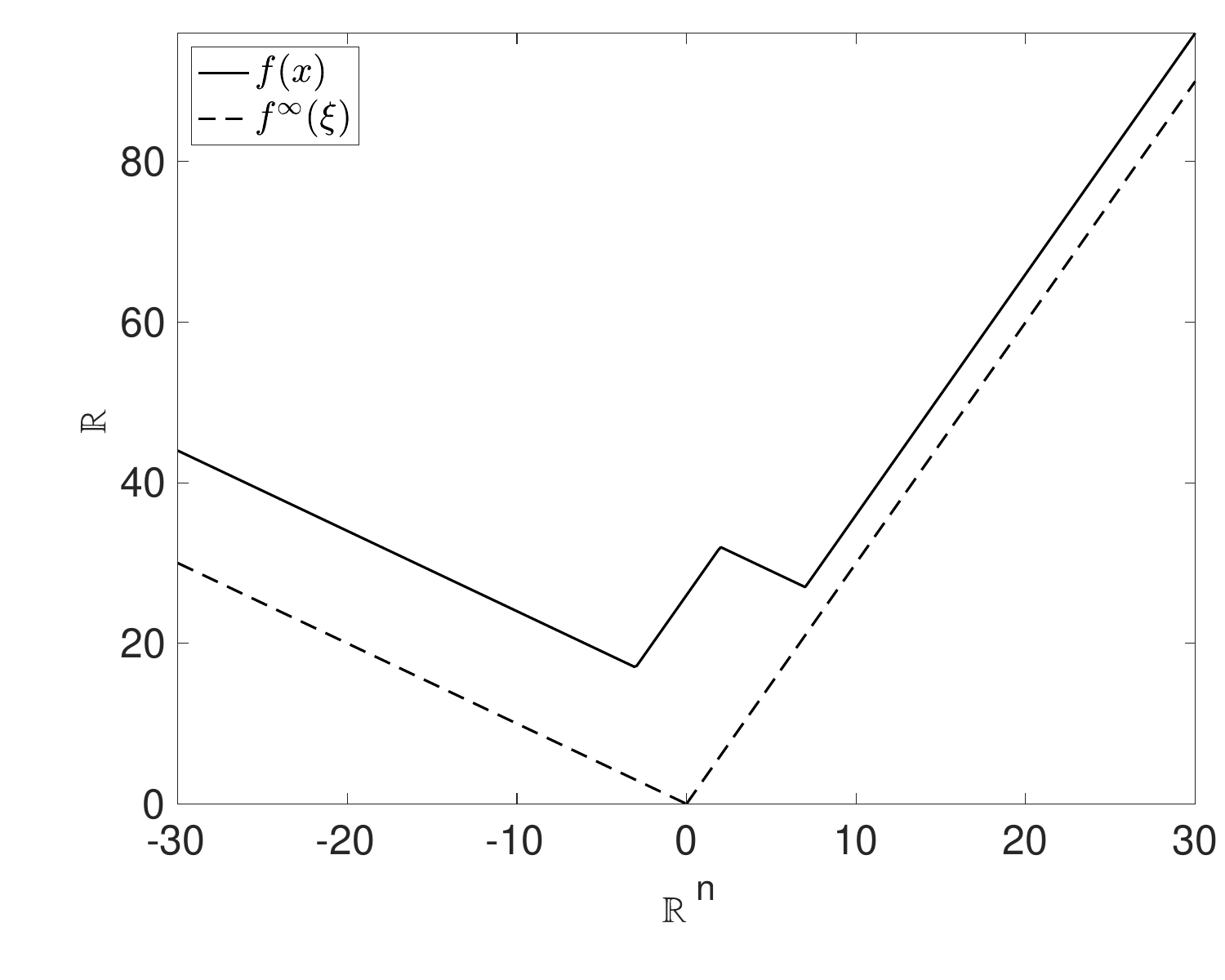}
    \caption{The horizon function $f^\infty$ (dashed) for the piecewise affine function $f$ given by \eqref{eq:horizon} (solid).}
    \label{fig:horizon}
\end{figure}

\begin{lemma}
\label{lem:infHF}
  Suppose that $f:\reals^n\to\reals$ is $\mathcal{PA}$, and choose
  a set $X\subset\reals^n$ for which $0\in\interior{X}$.  The function
  $f$ has a (global) minimum in $\reals^n$ if and only if
  $\inf_{\bm\xi\in X}f^\infty(\bm\xi)=0$.
\end{lemma}
\begin{proof}
  For the ``only if'' case, suppose that $\x^\ast\in\reals^n$ is a minimum of $f$, and let
  $\phi:=f(\x^\ast)$.  Thus, for each $\x\in\reals^n$, $\delta>0$, and
  $\lambda\in(0,\delta)$, we have
$\lambda f(\tfrac{\x}{\lambda}) \geq \lambda\phi \geq -\delta|\phi|$.
It follows that, for each $\bm\xi\in X$,
\begin{align*}
  f^\infty(\bm\xi)
  &= \lim_{\delta\downTo{0}}\left(\inf\{\lambda f(\tfrac{\x}{\lambda}):
      0<\lambda<\delta, \|\x-\bm\xi\|<\delta\}\right) \\
  &\geq \lim_{\delta\downTo{0}}(-\delta|\phi|) = 0 = f^\infty(\0),
\end{align*}
and so $\inf_{\bm\xi\in X}f^\infty(\bm\xi)=0$.

\yzedit{Next, for the ``if'' case, suppose that $\inf_{\bm\xi\in X}f^\infty(\bm\xi)=0$, which implies $f^\infty(\bm\xi)\geq 0$ for each $\bm\xi\in X$.} Since $f^\infty$ is positively homogeneous according to
\cite[Theorem~3.21]{RockafellarWets}, and since $0\in\interior{X}$, it
follows that $f^\infty(\bm\xi)\geq 0$ for each $\bm\xi\in\reals^n$.
By~\cite[Proposition~2.2.3]{Scholtes}, there exists a finite collection
of polyhedra $P^1,\ldots,P^p\subset\reals^n$ for which $f$ is affine
on each $P^i$.  Thus, for any $i\in\{1,\ldots,p\}$, the restriction of
$f$ to $P^i$ may be
described as $f^i:y\in P^i\mapsto\innerProd{\a^i}{\y}+b^i$ for some
$\a^i\in\reals^n$ and $\yzedit{b^i}\in\reals$.  By~\cite[Theorem~2.1.2]{Scholtes},
each $P^i$ may be expressed as a Minkowski sum $C^i+R^i$, where
$C^i\subset\reals^n$ is compact and $R^i$ (the \emph{recession cone}
of $P^i$) is a polyhedral cone.  

Suppose, to obtain a contradiction, that
$\innerProd{\a^i}{\br}<0$ for some $i$ and some $\br\in R^i$.  Then:
\begin{align*}
&f^\infty(\br)\\
&\quad\leq \lim_{\delta\downTo{0}}\left(\inf\{\lambda f(\tfrac{\x}{\lambda}): \right.\\
&\qquad\qquad\qquad\quad \left.0<\lambda<\delta,\quad \|\x-\br\|<\delta,\quad (\tfrac{\x}{\lambda})\in P^i\}\right) \\
&\quad= \lim_{\delta\downTo{0}}\left(\inf\{\innerProd{\a^i}{\x} +
  \lambda b^i:\right. \\
&\qquad\qquad\qquad\quad \left. 0<\lambda<\delta,\quad \|\x-\br\|<\delta,\quad (\tfrac{\x}{\lambda})\in P^i\}\right) \\
&\quad=\innerProd{\a^i}{\br}<0,
\end{align*}
which contradicts the assumption that $f^\infty(\bm\xi)\geq 0$ for each
$\bm\xi\in\reals^n$.  Thus, $\innerProd{\a^i}{\br}\geq 0$ for each $i\in\{1,\ldots,p\}$
and each $\br\in R^i$.  Observe, moreover, that $\0\in R^i$. Lastly,
since each $C^i$ is compact, there exists $\gamma^i\in\arg\min_{\x\in
  C^i}f(\x)$.  Combining these observations,
\begin{align*}
  &\inf_{\x\in\reals^n}f(\x) \\
  &=\min_{i\in\{1,\ldots,p\}}\left(b^i+\inf_{\cc\in
  C^i}\innerProd{\a^i}{\cc}+\inf_{\br\in
  R^i}\innerProd{\a^i}{\br}\right)\\
  &=\min_{i\in\{1,\ldots,p\}}\left(b^i + \innerProd{\a^i}{\bm\gamma^i} + 0\right) \\
  &=\min_{i\in\{1,\ldots,p\}} f(\bm\gamma^i).
\end{align*}
Thus, for some $i\in\{1,\ldots,p\}$, $\bm\gamma^i$ is a global minimum of $f$.
\end{proof}

The above lemma depends heavily on the $\mathcal{PA}$ assumption; it
is readily verified that the exponential mapping $x\mapsto e^x$ has a
nonnegative horizon function but does not have a minimum in
$\reals$.  The following lemma shows that horizon functions are
readily described for $\mathcal{PA}$ functions with known abs-normal forms.

\begin{lemma}
\label{lem:existOfmin}
    Suppose that $f:\reals^n \rightarrow \reals$ is $\mathcal{PA}$. The function $f$ has a global minimum in $\mathbb{R}^n$ if and only if the equation system $f^\infty(\bm\xi) + 1 = 0$ has no solution $\bm\xi\in\reals^n$.
\end{lemma}
\begin{proof}
  Since $f^\infty$ is positively homogeneous, there exists $\x\in\reals^n$ for which $f^\infty(\x)<0$ if and only if there exists $\bm\xi\in\reals^n$ for which $f^\infty(\bm\xi)+1=0$. At this point, the lemma follows immediately from Lemma~\ref{lem:infHF}.
\end{proof}

\begin{lemma}
\label{lem:absHF}
  Suppose that Assumption~\ref{ass:ANF} holds with $m=1$. The horizon function
  $f^\infty$ is $\mathcal{PL}$, and has the abs-normal form:
  \[
\begin{bmatrix}
  \bm\zeta \\ f^\infty(\bm\xi)
\end{bmatrix}
\equiv
\begin{bmatrix}
  \Z & \LL \\ \J & \Y
\end{bmatrix}
\begin{bmatrix}
  \bm\xi \\ |\bm\zeta|
\end{bmatrix}.
  \]
\end{lemma}
\yzedit{(Here, $\bm\zeta/\bm\xi$ play the same role as $\z/\x$ did in~\eqref{def:ANF}. We use differing notation here to distinguish a $\mathcal{PA}$ function $f$ from its horizon function $f^\infty$).}
\begin{proof}
  The piecewise linearity of $f^\infty$ was established by Scholtes~\cite[Proposition~2.5.1]{Scholtes}; Scholtes calls the horizon function the ``recession function''.

Now, for any $\x\in\reals^n$, let $\z(\x)$ denote the unique value of $\z$ that is consistent with \eqref{eq:ANF}. Suppose $\x\in\reals^n$ is fixed. Altering notation for consistency, Scholtes notes that there is some small $\nu>0$ for which, for each $\lambda\in\reals$ with $\nu>\lambda>0$,
  \[
    \lambda f(\tfrac{\x}{\lambda}) = \transpose{\mathbf{a}}\x,
  \]
  where $\y\mapsto\transpose{\mathbf{\a}}\y+b$ is a particular affine piece of $f$ (independent of $\lambda$) that is active at $\frac{1}{\nu}\x$. Hence, there is some fixed diagonal matrix $\bm{\Sigma}\in\reals^{s\times s}$ for which $|\z(\frac{1}{\lambda}\x)|=\bm{\Sigma}\,\z(\frac{1}{\lambda}\x)$ whenever $\nu>\lambda>0$. At this point, if we start with \eqref{eq:ANF}, replace $\x$ with $\frac{1}{\lambda}\x$ and $\z$ with $\z(\frac{1}{\lambda}\x)$,   note that $|\z(\frac{1}{\lambda}\x)|=\bm\Sigma\,\z(\frac{1}{\lambda}\x)$ for all sufficiently small $\lambda>0$, and take the limit $\lambda\to 0^+$, we obtain \yzedit{this lemma's claim.}
\end{proof}

As described by Scholtes~\cite{Scholtes}, given a $\mathcal{PA}$
function $f:\reals^n\to\reals$ and some $\x\in\reals^n$, the directional derivative mapping:
\[
f'(\x;\cdot):d\mapsto\lim_{t\downTo{0}}\tfrac{1}{t}(f(x+td)-f(x))
\]
 is also a
positively homogeneous $\mathcal{PL}$ approximation of $f$.
Note, however, that $f'(\x;\cdot)$ is not the same as $f^\infty$.
While, roughly,
$f'(\x;\cdot)$ describes the behavior of $f$ near $\x$, $f^\infty$
instead describes the behavior of $f$ far from the origin.

The above lemmata can be combined to characterize existence of global minima for a $\mathcal{PA}$ function in abs-normal form, as follows.

\begin{corollary}
    Suppose that Assumption~\ref{ass:ANF} holds with $m=1$, and consider the auxiliary matrices and vectors from Definition~\ref{def:MandV}. The function $f$ has a global minimum in $\reals^n$ if and only if there is no solution $(\bm\xi,\bm\omega)$ of the following MLCP:
\begin{equation}
\label{eq:auxMLCP}
    \begin{array}{rcl}
        0 &=& 1 + \tilde{\J}\bm\xi + \tilde{\Y}\bm\omega, \\
        \0 \leq \bm\omega &\perp & \tilde{\Z}\bm\xi + \tilde{\LL}\bm\omega \ge \0.
    \end{array}
\end{equation}
\end{corollary}

\begin{proof}
This claim follows immediately from Theorem~\ref{thm:newMLCP}, Lemma~\ref{lem:existOfmin}, and Lemma~\ref{lem:absHF}.
\end{proof}

\yzedit{For instance, continuing Example~\ref{ex:construction}, we can verify the existence of a global minimum of $f$ by verifying that the following MLCP, formulated using the matrices and vectors from Example~\ref{ex:ConstructMartixAndVector}, has no solution:
\begin{align*}
0 &= 1 +  49 \xi + \begin{bmatrix}
    4 & 2 & 12
\end{bmatrix}\begin{bmatrix}
      \omega_1 \\ \omega_2 \\ \omega_3
  \end{bmatrix} \\
  \0\leq \begin{bmatrix}
      \omega_1 \\ \omega_2 \\ \omega_3
  \end{bmatrix} &\perp \begin{bmatrix}
    3 \\ 6 \\ 7
  \end{bmatrix}[\xi] + \begin{bmatrix}
      1 & 0 & 0 \\
      4 & 1 & 0\\
      0 & 0 & 1 
  \end{bmatrix}\begin{bmatrix}
      \omega_1 \\ \omega_2 \\ \omega_3
  \end{bmatrix} \geq \0,
\end{align*}}

\section{Implementation}
\label{sec:implementation}

We have developed a proof-of-concept Julia implementation of our new root-finding and minimization approaches for $\mathcal{PA}$ functions.  This implementation has been uploaded to GitHub\footnote{Available at \url{https://github.com/kamilkhanlab/abs-normal}} under the MIT license. In this implementation, for any user-defined $\mathcal{PA}$ function provided in the abs-normal form specified by Assumption~\ref{ass:ANF}, the corresponding auxiliary quantities in Definition~\ref{def:MandV} are computed automatically.

The root-finding systems MLCP~\eqref{eq:newMLCP} and LCP~\eqref{eq:newLCP} (when $\tilde{\J}$ is nonsingular) can then be automatically constructed and solved using either the \textsf{PATHSolver.jl} v1.7.5 interface~\cite{PathSolver} for the \textsf{PATH} complementarity solver \cite{DirkseFerris, FerrisMunson}, or the global optimizer \textsf{BARON.jl} v0.8.3~\cite{sahinidis1996baron}. For comparison, we also implemented the root-finding approach of Griewank et al.~\cite{GriewankStreubel}, with their MLCP formulation~\eqref{eq:oldMLCP} (when $\J$ is nonsingular) and LCP formulation~\eqref{eq:oldLCP} (when $\I-\bS$ is further nonsingular)  automatically constructed and solved using \textsf{BARON.jl} and \textsf{PATHSolver.jl}, respectively.

To verify the existence of a global minimum of a scalar-valued $\mathcal{PA}$ function $f$, our implementation automatically constructs and solves the MLCP~\eqref{eq:auxMLCP} using \textsf{BARON.jl} v0.8.3~\cite{sahinidis1996baron}. 
\yzedit{If $f$ indeed has a global minimum, both LPCC~\eqref{eq:LPCC} and MILP~\eqref{eq:MILP} systems can be used to minimize $f$. Here, the LPCC~\eqref{eq:LPCC} and MILP~\eqref{eq:MILP} (with $\mu:=10^5$ by default) are both solved using \textsf{BARON.jl}.} All optimization problems in this implementation are formulated using \textsf{JuMP.jl} v1.14.1~\cite{Lubin2023}. 

All numerical examples described below were conducted on a Windows 11 machine with a 2.50 GHz Intel i5-13400 CPU and 16GB RAM. 

\section{Numerical examples}
\label{sec:examples}

This section presents four numerical examples to illustrate our new approaches. First, the following example identifies a root of a $\mathcal{PA}$ function. 

\begin{Example}\label{ex:RootFinding}
Consider a $\mathcal{PA}$ function:
\begin{align*}
    &f_1(x_1,x_2) = \Bigl|\left|x_1+2\right|+x_2-1\Bigr| - x_2 -1,\\
    &f_2(x_1,x_2) = \left|x_1+2\right| + 2x_2-1,
\end{align*}
with the abs-normal form~\eqref{eq:ANF} where:
\yzedit{
\begin{align*}
    \cc&= \begin{bmatrix}
    \phantom{-}2 \\
    -1 
    \end{bmatrix},
    & \Z&= \begin{bmatrix}
    1 & 0\\
    0 & 1
    \end{bmatrix},
    & \LL&= \begin{bmatrix}
    0 & 0\\
    1 & 0
    \end{bmatrix},\\
    \bb&= \begin{bmatrix}
    -1\\
    -1
    \end{bmatrix}, 
    &\J&= \begin{bmatrix}
    0 & -1\\
    0 & \phantom{-}2
    \end{bmatrix},
    &\Y&= \begin{bmatrix}
    0 & 1\\
    1 & 0
    \end{bmatrix}.
\end{align*}
}
Using our Julia implementation of the root-finding methods described in Section 2, $\f(\x)=\0$  was solved to yield $(x_1,x_2)=(0,-0.5)$, via both the MLCP~\eqref{eq:newMLCP} and the LCP~\eqref{eq:newLCP}.
\end{Example}

Next, given a $\mathcal{PA}$ function $f:\mathbb{R}^n \rightarrow \mathbb{R}$, we demonstrate that the MILP system~\eqref{eq:MILP} and LPCC system~\eqref{eq:LPCC} effectively identify a global minimum. This minimum's existence is verified through the MLCP system~\eqref{eq:auxMLCP}.

\begin{Example}\label{ex:optimization}
Consider a $\mathcal{PA}$ function:
\[
    f(x) = \Bigl|x_1+ \left|2x_2-1\right| \Bigr|+ \left|3 + x_3\right|
\]
whose abs-normal form is based on the following vectors and matrices:
\yzedit{
\begin{align*}
    \cc&= \begin{bmatrix}
    -1 \\ \phantom{-}0  \\ \phantom{-}3
    \end{bmatrix},
    & \Z&= \begin{bmatrix}
    0 & 2 & 0\\
    1 & 0 & 0\\
    0 & 0 & 1
    \end{bmatrix},
    & \LL&= \begin{bmatrix}
    0 & 0 & 0\\
    1 & 0 & 0\\
    0 & 0 & 0\\
    \end{bmatrix},\\
    \bb&= \begin{bmatrix}
    0
    \end{bmatrix}, 
    &\J&= \begin{bmatrix}
    0 & 0 & 0\\
    \end{bmatrix},
    &\Y&= \begin{bmatrix}
    0 & 1 & 1
    \end{bmatrix}.
\end{align*}
}
The existence of a global minimum for $f(\x)$ was verified using our implementation of MLCP~\eqref{eq:auxMLCP}, showing that $f(\x)$ indeed has a global minimum. Then, the minimum was numerically computed as $\x^\ast=(0.0, 0.5, \yzedit{-3.0})$ using our Julia implementation of both the MILP~\eqref{eq:MILP} and the LPCC~\eqref{eq:LPCC}.
\end{Example}

The following example describes how the computation time for determining a root of $\f:\reals^n\to \reals^n$ scales with $n$, via both the MLCP~\eqref{eq:newMLCP} and the LCP~\eqref{eq:newLCP}. 

\begin{Example}\label{ex:ScalableRootFinding}

    Consider a  $\mathcal{PA}$ function $\f:\reals^n\to\reals^n$ in abs-normal form~\eqref{eq:ANF} \yzedit{with $s=n$, randomly generated as follows}.
    All components of  $\cc$, $\bb$, and \yzedit{$\Y$ were generated based on a normal distribution and rounded to the nearest integer, $\J$ is set to $\I$, and $\Z$ is set to $\mathbf{0}$.} \yzedit{The strictly lower triangular matrix $\mathbf{L}$ is set to have ones on the first lower subdiagonal, with all other elements being zero.}
    For several values of $n$, ranging from $2$ to $500$, $\tilde{\J}$ was verified to be nonsingular, and we solved $\f(\x)=\0$ by applying our Julia implementations of the MLCP~\eqref{eq:newMLCP} and the LCP~\eqref{eq:newLCP}, \yzedit{employing both \textsf{PATHSolver.jl} and \textsf{BARON.jl} as solvers.} Figure~\ref{fig:cpu_time_root_finding} depicts the corresponding CPU times, averaged over 100 runs. As expected, under both formulations, solution time increases as the problem size increases. Observe that for larger values of $n$, the LCP formulation is solved much faster than the MLCP formulation. \yzedit{Although we observe that \textsf{PATHSolver.jl} may fail to solve both LCP and MLCP formulations at larger values of $n$,  when it is successful, it consistently requires fewer CPU times than the general-purpose global optimizer \textsf{BARON.jl} for both formulations, especially as the problem size increases. As highlighted in the zoomed-in section in Figure~\ref{fig:cpu_time_root_finding}, for smaller values of $n$, solving the MLCP with \textsf{PATHSolver.jl} is even faster than solving the LCP with \textsf{BARON.jl}.}
    
\begin{figure}[h]
    \centering
    \includegraphics[width=0.7\linewidth]{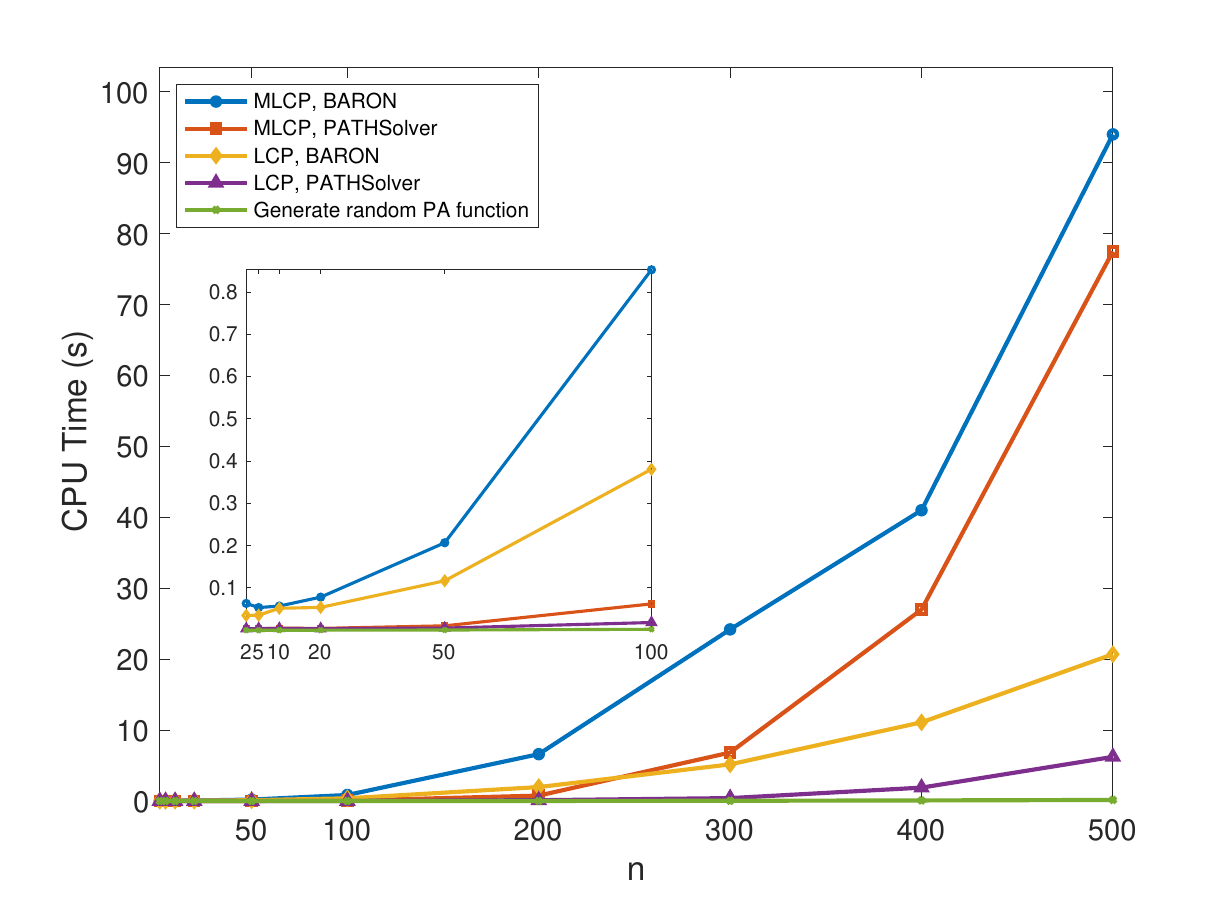}
    \caption{\yzedit{CPU time (s) for generating a random $\mathcal{PA}$ function in Example~\ref{ex:ScalableRootFinding} and determining its root using our new root-finding approaches (including computing all necessary vectors, matrices, and solving the corresponding formulations), averaged over 100 runs. The inset figure zooms in on the lower-left corner.}}
    \label{fig:cpu_time_root_finding}
\end{figure}
\end{Example}

\yzedit{In the following example, we compare our new root-finding formulations MLCP~\eqref{eq:newMLCP} and LCP~\eqref{eq:newLCP} with the approach of Griewank et al.~\cite{GriewankStreubel}, where we analyze the computation time required for determining a root of $\f: \reals^n \rightarrow \reals^n$ as it scales with $n$.}
\yzedit{
\begin{Example}\label{ex:CompareRootGriewank}
Consider a randomly generated $\mathcal{PA}$ function $\f:\reals^n\to\reals^n$ in abs-normal form, where $s=n$. We generate coefficient vectors $\cc$ and $\bb$, and the coefficient matrices $\Y$, $\Z$, and $\LL$, in the same manner as in Example~\ref{ex:ScalableRootFinding}. Here, the coefficient matrix $\J$ is set as the identity matrix since Griewank et al.'s MLCP~\eqref{eq:oldMLCP} and LCP~\eqref{eq:oldLCP} systems require $\J$ to be nonsingular. For several values of $n$, ranging from $2$ to $100$, $\J$ and $(\I-\bS)$ were verified to be nonsingular, and we solved $\f(\x) = \0$ using our Julia implementation of the new root-finding approaches with \textsf{PATHSolver.jl}, alongside Griewank et al.’s approaches. Here, results for larger $n$ values ($n \geq 100$) are not depicted, as \textsf{PATHSolver.jl}  failed to solve the corresponding formulations. Figure~\ref{fig:cpu_time_compare_with_Griewank} shows the CPU times averaged over 100 runs. The small inset in Figure~\ref{fig:cpu_time_compare_with_Griewank} excludes the computation time for solving the previous MLCP~\eqref{eq:oldMLCP} using \textsf{BARON.jl}, allowing for a clearer comparison between the remaining formulations. Observe that our new LCP~\eqref{eq:newLCP}, solved by \textsf{PATHSolver.jl}, is faster than other approaches across all values of $n$. 
\begin{figure}[h]
    \centering
    \includegraphics[width=0.7\linewidth]{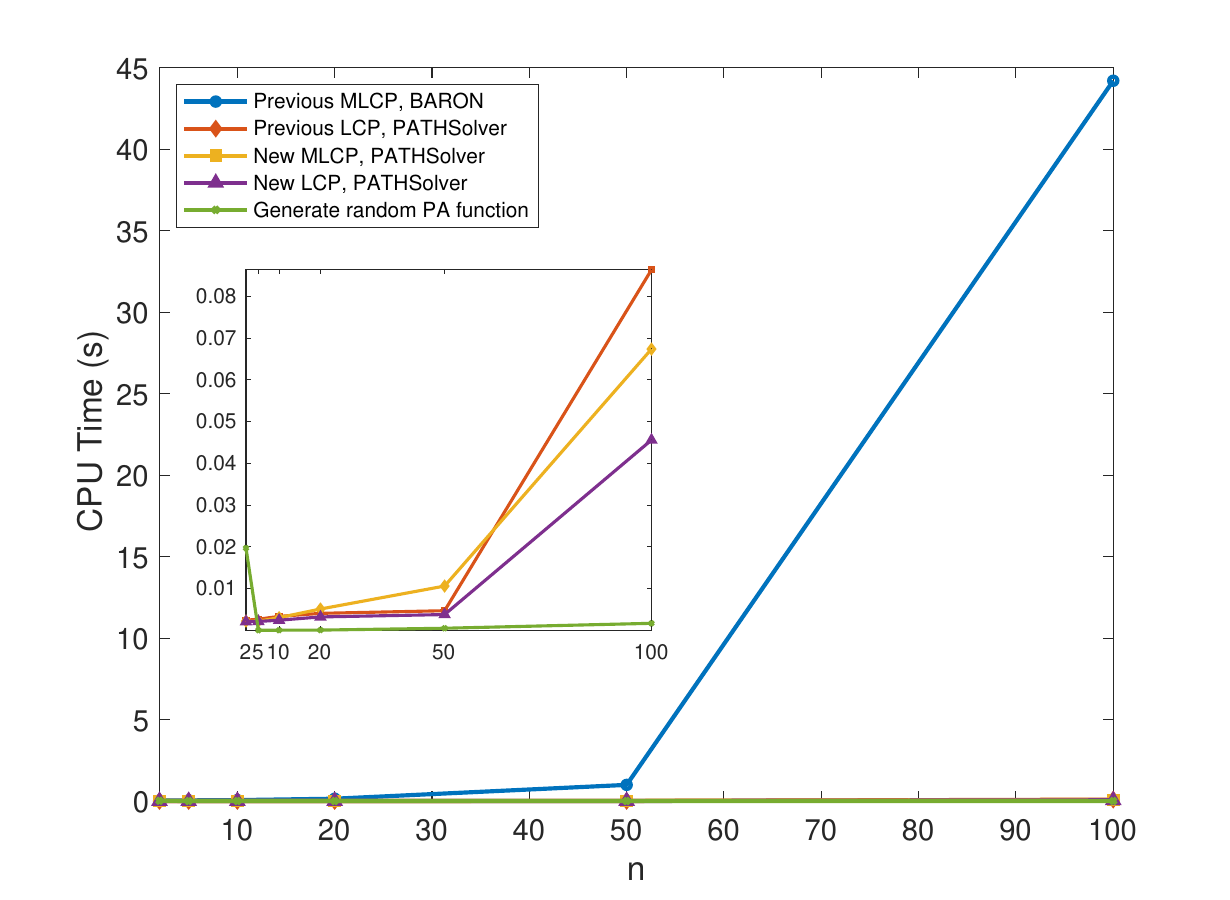}
    \caption{\yzedit{CPU time (s) for generating a random $\mathcal{PA}$ function in Example~\ref{ex:CompareRootGriewank} and determining its root using both our new approaches and Griewank et al.'s approaches (including computing all necessary vectors, matrices, and solving the corresponding formulations), averaged over 100 runs. The inset figure zooms in on the lower-left corner.}}\label{fig:cpu_time_compare_with_Griewank}
\end{figure}    
\end{Example}
}

Our final example will illustrate the LPCC~\eqref{eq:LPCC} and MILP~\eqref{eq:MILP} formulations for minimizing a $\mathcal{PA}$ function. 

\begin{Example}\label{ex:ScalableOptimization} Consider the following $\mathcal{PA}$ function $f:\reals^n \rightarrow \reals$ in abs-normal form, with the dimension $n$ to be varied for illustration:
\begin{align*}
  f(\x) &= \biggl| \Bigl| \bigl| |1000x_1| + 1000x_2 \bigr| + 1000x_3 \Bigr| + \dots + 1000x_n \biggr| \\
  &\qquad{}+ 1.0,
\end{align*}
\yzedit{In this example, the number of absolute value functions present is equal to the dimension $n$ of the input $\x$}. For several choices of $n$, ranging from $2$ to $500$, \yzedit{the existence of a global minimum was verified using our implementation of MLCP~\eqref{eq:auxMLCP}}, and computed numerically using our implementation of the LPCC~\eqref{eq:LPCC} and the MILP~\eqref{eq:MILP}. The corresponding CPU times (averaged over \yzedit{100} runs) are depicted in Figure \ref{fig:cpu_time_optimal_finding}, \yzedit{with a zoomed-in portion showing the behavior for smaller values of $n$.} Observe that the MILP~\eqref{eq:MILP} was faster than the LPCC~\eqref{eq:LPCC} in all of these instances.

\begin{figure}[h]
    \centering
    \includegraphics[width=0.7\linewidth]{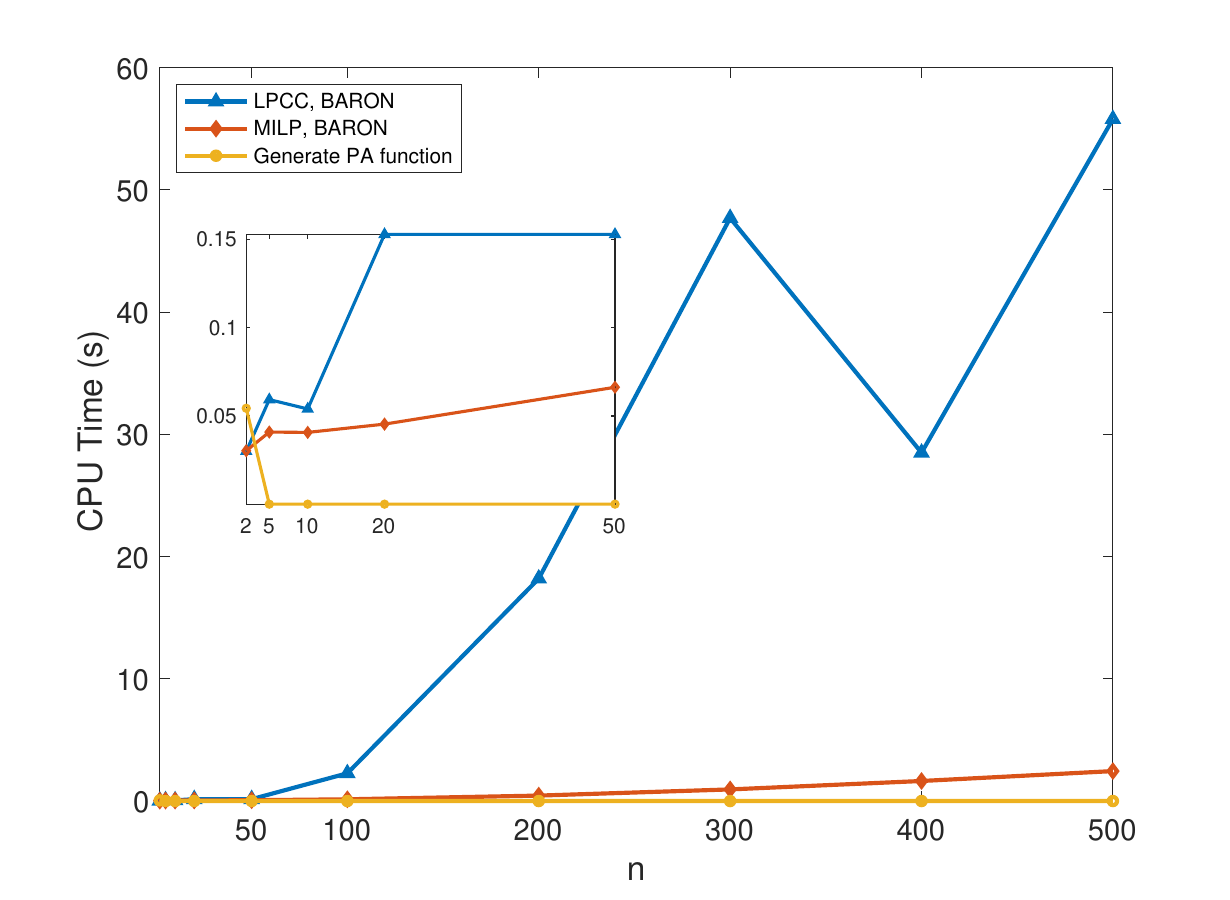}
    \caption{\yzedit{
    CPU time (s) for computing our auxiliary matrices/vectors in Definition~\ref{def:MandV} for the function $f$ in Example~\ref{ex:ScalableOptimization} (yellow), and for minimizing it via our MILP formulation (red) and our LPCC formulation (blue), averaged over 100 runs. The inset figure zooms in on the lower-left corner.}}
    
    \label{fig:cpu_time_optimal_finding}
\end{figure}

\end{Example}

\section{Conclusion}   
The above root-finding and optimization approaches are essentially new post-processing steps for the abs-normal form, and contribute to the practical utility of deploying the abs-normal form when considering ``real-world'' $\mathcal{PA}$ functions. Recall that the abs-normal form itself is constructed via an AD-like procedure, and is also based on \yzedit{another} piecewise-linearization AD procedure when a $\mathcal{PA}$ approximation of a nonsmooth function is sought. Now that piecewise-affine root-finding and optimization have been expressed as standard complementarity problems or LPCCs, we expect that any future advances in solving LCPs, MLCPs, MILPs, and/or LPCCs will also \yzedit{ enhance our capacity for computations involving $\mathcal{PA}$ functions.} 
\yzedit{Future work may involve tailoring these methods to special cases, such as abs-normal forms that naturally represent convex functions or difference-of-convex functions (such as by~\cite{kazda2021nonconvex}).
Future work may involve applying these methods to nontrivial application-oriented examples, perhaps involving neural networks, though this would require automation of the relevant function’s abs-normal form. (Our current implementation requires the abs-normal form as an input, and these were computed by hand in our numerical examples.) Since there is a new Julia interface \textsf{ADOLC.jl} v1.3.0~\cite{ADOLC.jl} to the AD package \textsf{ADOL-C}~\cite{ADOLC}, and since ADOL-C contains an implementation of piecewise linearization, perhaps these approaches could be chained together.}

\bibliographystyle{siam} 
\bibliography{references}

\end{document}